\documentclass{amsart}
\usepackage{amssymb,amsthm, amsmath, amsfonts, amsbsy}
\usepackage{graphics}
\usepackage{hyperref}
\usepackage[all]{xy}
\usepackage{enumerate}
\usepackage[mathscr]{eucal}
\usepackage{bbm}
\usepackage{indentfirst}

\theoremstyle{plain}
\newtheorem{theorem}{Theorem}[section]
\newtheorem{lemma}[theorem]{Lemma}
\newtheorem{proposition}[theorem]{Proposition}

\newtheorem{corollary}[theorem]{Corollary}
\numberwithin{equation}{section}

\theoremstyle{definition}

\newtheorem{definition}[theorem]{Definition}
\newtheorem{example}[theorem]{Example}
\newtheorem{remark}[theorem]{Remark}

\DeclareMathOperator{\Mod}{-Mod}
\DeclareMathOperator{\module}{-mod}

\DeclareMathOperator{\hd}{hd}
\DeclareMathOperator{\gd}{gd}
\DeclareMathOperator{\Ob}{Ob}
\DeclareMathOperator{\pd}{pd}

\DeclareMathOperator{\gl.dim}{gl.dim}

\DeclareMathOperator{\Char}{char}

\newcommand{\FI}{{\mathscr{FI}}}

\title[Global dimensions of category algebras of finite categories]{Classifying global dimensions of category algebras of some finite combinational categories}

\author{Liping Li}
\address{HPCSIP (Ministry of Education), College of Mathematics and Statistics, Hunan Normal University; Changsha, Hunan 410081, China.}
\email{lipingli@hunnu.edu.cn}

\author{ting zhang}
\address{College of Mathematics and Statistics, Hunan Normal University; Changsha, Hunan 410081, China.}
\email{201710100078@hunnu.edu.cn}

\thanks{The first author is supported by the National Natural Science Foundation of China 11771135 and the Start-Up Funds of Hunan Normal University 830122-0037.}

\begin{document}

\begin{abstract}
In this paper we classify global dimensions of the category algebras of some finite categories, including finite truncations of the categories $FI$, $FI_G$, $FI_d$, $OI$, $OI_G$, $OI_d$ and $VI$.
\end{abstract}

\maketitle
\section{Introduction}

In this paper we study representation theory of some combinatorial categories, an active area attracting the attention of a lot people because of its numerous applications in other areas such as representation stability theory; see for instances \cite{CEF, CEFN, D, GL, GL1, LR, LY, N, PS, SS, St, St1, St2}. Specifically, since the category algebras (over a field) of many infinite combinatorial categories have infinite global dimension; see for instances \cite[Thorem 4.9]{LY} for $FI$, the category of finite sets and injections, \cite[Corollary 1.6]{GL1} and \cite[Theorem 1.9]{N} for $VI$, the category of finite dimensional vector spaces over finite fields and linear injections, and \cite[Corollary 5.18]{SS2} for $FI_d$, the category of finite sets and pairs of injections and $d$-coloring maps, we are motivated to establish the finite version of this result. That is, we study global dimensions of category algebras (over a field) of the following finite combinatorial categories, which (up to category equivalence) are truncations of those infinite combinatorial categories appearing in representation stability theory:

\begin{center}
Table 1: A few finite combinatorial categories
\begin{tabular}{c|cc}
  Categories & Objects & Morphisms\\
  \hline
  $FI_n$ & $S \subseteq [n]$ & injections\\
  $FI_{G, n}$ & $S \subseteq [n]$ & $(f, g): S \times S \to T \times G$ where $f$ is injective\\
  $FI_{d,n}$ & $S \subseteq [n]$ & pairs of injections and $d$-coloring maps\\
  $OI_n$ & $S \subseteq [n]$ & order-preserving injections\\
  $OI_{G,n}$ & $S \subseteq [n]$ & $(f, g): S \times S \to T \times G$ where $f$ is injective and order-preserving\\
  $OI_{d,n}$ & $S \subseteq [n]$ & pairs of order-preserving injections and $d$-coloring maps\\
  $VI_n$ & $F_q^{\oplus S}, \, S \subseteq [n]$ & linear injections.
\end{tabular}
\end{center}

Our main result is:

\begin{theorem} \label{main theorem}
Let $n$ be a fixed natural number (including 0), and let $k$ be a field. Then one has:
\begin{center}
\begin{tabular}{c|c}
  Category algebras & Global dimensions\\
  \hline
  $kFI_n$ & $\begin{cases} \infty, & \text{if } \Char k \mid n!; \\ n, & \text{else}.\end{cases}$\\
  $kFI_{G, n}$ & $\begin{cases} \infty, & \text{if } \Char k \mid n!|G|^n; \\ n, & \text{else}.\end{cases}$\\
  $kFI_{d,n}$ & $\begin{cases} \infty, & \text{if } \Char k \mid n!; \\ n, & \text{else}.\end{cases}$ \\
  $kOI_n$ & $n$. \\
  $kOI_{G, n}$ & $\begin{cases} \infty, & \text{if } \Char k \mid |G|; \\ n, & \text{else}.\end{cases}$\\
  $kOI_{d,n}$ & $n$. \\
  $kVI_n$ & $\begin{cases} \infty, & \text{if } \Char k \mid (q^n-1)(q^n-q) \ldots (q^n-q^{n-1}); \\ n, & \text{else}.\end{cases}$
\end{tabular}
\end{center}
\end{theorem}

We briefly describe the essential idea to establish the above classification, which is different from the combinatorial representation theory approach described in \cite{St1, St2}, where Stein and Steinberg classified the global dimensions of some important finite monoids (over the complex field). Firstly, since these finite combinatorial categories are equivalent to \emph{truncations} of corresponded infinite combinatorial categories respectively, we may view a representation a finite combinatorial category as a representation of the corresponded infinite combinatorial category via the lift process. Secondly, there are natural restriction functors from categories of representations of infinite combinatorial categories to categories of representations of truncations, and we can show that these restriction functors preserve projective covers and hence minimal projective resolutions. Since there already exists a clear description about minimal projective resolutions for some special representations (including those lifted from representations of truncations) of these infinite combinatorial categories in literature (see for instance \cite{L}), applying the restriction functor, we deduce the above result.

This paper is organized as follows. In Section 2 we describe some preliminary results on representations of categories. The lift functor and restriction functor are introduced in Section 3, where we prove that the restriction functor preserves projective covers and minimal projective resolutions. In the last section we prove the main theorem.

\section{Preliminaries}

\subsection{Representations of categories}

Throughout this paper let $k$ be a field and let $C$ be a small category. We briefly recall some background on representations of categories; for more details, please refer to \cite{M, W, Xu}.

\begin{definition}
A $k$-linear representation of $C$ (or a $C$-module) is a covariant functor $V$ from $C$ to $k\Mod$, the category of all vector spaces. A \emph{locally finite} $k$-linear representation of $C$ is a covariant functor from $C$ to $k \module$, the category of finite dimensional vector spaces.
\end{definition}

Denote the category of $C$-modules by $C\Mod$. By definition, a $C$-module $V$ is locally finite if and only if for every object $x$ in $C$, the value of $V$ on $x$, denoted by $V_x$, is a finite dimensional vector space. In this paper we only consider locally finite representations. Since kernels and cokernels are defined pointwise, the category of locally finite representations is abelian.

We can also study representations of $C$ from the viewpoint of module theory. For this purpose, we need the following definition of category algebras, which was introduced by Webb in \cite{W}.

\begin{definition}
Let $C$ be a small category. The category algebra $kC$ is the free $k$-module whose basis is the set of morphisms in $C$. The multiplication is induced by the following rule on the basis elements of $kC$ by letting
\begin{equation*}
f \ast g = \begin{cases}
 f \circ g, & \text{if f and g can be composed in}\ C;\\
 0, & \text{otherwise} .
\end{cases}
\end{equation*}
\end{definition}

With this multiplication, $kC$ becomes an associative $k$-algebra. Let $kC\Mod$ be the category of $kC$-modules and $kC \module$ be the category of finitely generated $kC$-modules. The following theorem illustrates a relationship between $kC$-modules and $C$-modules; see \cite{M}.

\begin{theorem}
For any small category $C$, there exist functors $\iota : C \Mod \to kC\Mod$ and $\sigma : kC\Mod \to C \Mod$ such that $\iota$ is fully faithful, and $\sigma \circ \iota \cong \mathrm{Id}_{C\Mod}.$ Moreover if $\Ob (C)$, the set of objects in $C$, is finite, then $\iota \circ \sigma \cong \mathrm{Id}_{kC\Mod}$.
\end{theorem}

Consequently, $C\Mod$ can be identified with a full subcategory of $kC \Mod$, and every $C$-module can be viewed as a $kC$-module.

A small category $C$ is called a \emph{weakly directed category} if there is a preorder $\preccurlyeq$ on $\Ob (C)$ such that for two objects $x$ and $y$, one has $x \preccurlyeq y$ if $C (x, y)$, the set of morphisms from $x$ to $y$, is not empty. For weakly directed categories, following \cite[Definition 4.2.18]{Xu}, one can define \emph{ideals} and \emph{coideals} as follows:

\begin{definition}
Let $D$ be a a full subcategory of $C$. We say that $D$ is an ideal in $C$ if for any $x \in \Ob (D)$ and $y \in \Ob (C)$, whenever $y \preccurlyeq x$ one has $y \in \Ob (D)$ as well. Dually, we define coideals.
\end{definition}

If $C$ is an \emph{EI-category}; that is, every endomorphism in $C$ is an isomorphism, then the reader can check that $C$ is weakly directed. Furthermore, the preorder defined above induces a partial order on the set of isomorphism classes of objects. For details, see \cite{W}.

\subsection{Some combinatorial categories and their representations}

For the convenience of the reader, let us repeat the definitions of infinite combinatorial categories considered in this paper. For more information, one may refer to \cite{GL, SS}.

\begin{example}[The category $FI$]
Objects of the category $FI$ are finite sets of natural numbers,\footnote{In literature, objects of $FI$ are finite sets rather than finite sets of natural numbers. The reason that we impose this extra assumption is that we are working with small categories. That is, the class of objects must be a set. This extra assumption does not matter since two categories are equivalent. By the same reason, our definitions of other categories are also slightly different the ones appearing in literature.} and morphisms of $FI$ are injections between them.
\end{example}

\begin{example}[The category $FI_G$]
Let $G$ be a finite group. Objects of the category $FI_G$ are finite sets of natural numbers. For two finite sets $S$ and $T$, morphisms between them are pairs of map $(f, g): S \times S \to T \times G$, where $f$ is an injection and $g$ is any map. The composition of $(f_1, g_1)\in \FI_G (S, T)$ and $(f_2, g_2)\in FI_G (T,U)$ is defined by
\begin{equation*}
(f_2, g_2) (f_1, g_1) = (f_3, g_3)
\end{equation*}
where
\begin{equation*}
f_3 (x)=f_2(f_1(x)), \quad g_3(x)=g_2(f_1(x))g_1(x), \quad \mbox{ for all } x \in S.
\end{equation*}
When $G = 1$, $FI_G$ coincides with $FI$.
\end{example}

\begin{example}[The category $FI_{d}$]
Let $d$ be a positive integer. Objects of $FI_d$ are finite sets of natural numbers. For two finite sets $S$ and $T$, the morphism set $FI_d (S, T)$ is the set of all pairs $(f,\delta)$ where $f: S \to T$ is an injection, and $\delta: T \setminus f(S) \to [d]$ is an arbitrary map. The composition of $(f_1, \delta_1) \in FI_d(S, T)$ and $(f_2, \delta_2) \in FI_d(T,U)$ is defined by
\begin{equation*}
(f_2, \delta_2) (f_1, \delta_1) = (f_3, \delta_3)
\end{equation*}
where $f_3 = f_2 f_1$ and
\begin{equation*}
\delta_3 (x)= \left\{ \begin{array}{ll}
\delta_1(r) & \mbox{ if } x = f_2(r) \mbox{ for some } r, \\
\delta_2(x) & \mbox{ else. } \end{array} \right.
\end{equation*}
\end{example}

\begin{example}[The categories $OI$, $OI_G$, and $OI_d$]
The category $OI$ is a subcategory of $FI$ sharing the same objects, while morphisms between objects are injections preserving the natural order on the set of natural numbers. The categories $OI_G$ and $OI_d$ are the corresponded subcategory of $FI_G$ and $FI_d$ respectively.
\end{example}

\begin{example}[The categories $VI$]
Let $q$ be a power of a certain prime and $F_q$ be a finite field. Objects of $VI$ are vector spaces over $F_q$ whose bases are finite sets of natural numbers. Morphisms between objects are linear injections.
\end{example}

In the rest of this paper we denote by $C$ any one of the above mentioned categories unless other specified. It is easy to see that $C$ is a weakly directed category with respect to the natural inclusions of finite sets or finite dimensional vector spaces. It is also an EI-category since the endomorphisms of objects form groups (symmetric groups for $FI$ and $FI_d$, wreath products of symmetric groups and $G$ for $FI_G$, trivial group for $OI$ and $OI_d$, the direct product of $n$ copies of $G$ for $OI_G$ (via identifying a map from $[n]$ to $G$ with an element in $G^n$), and general linear groups for $VI$).

It has been proved that $C$ is locally Noetherian over any commutative Noetherian ring in \cite{SS}; that is, the category of finitely generated $C$-modules is abelian. Therefore, from now on we only consider finitely generated $C$-modules. We also point out that the category of $kC$-modules has enough projectives. In particular, for every object $x$ in $C$, the $kC$-module $kC e_x$ is a projective module, where $e_x$ is the identity morphism on $x$. Since $kC e_x$, when viewed as a representation of $C$, is the representable functor $kC(x, -)$, one deduces that finitely generated $kC$-modules and finitely generated $C$-modules coincide.

An important combinatorial property of $C$ is that it has a self-embedding functor $\iota: C \to C$, which we describe briefly here. The reader can refer to \cite{GL} for more details. Let $\ast$ be an element which is not a number. Then for every finite set $S$ of natural numbers, the natural inclusion $S \to S \cup \{ \ast \}$ or $F^S \to F^{S \cup \{ \ast \}}$ induces a fully faithful functor $\iota'$ from $C$ to another category $D$, which has objects subsets of $N \cup \{ \ast \}$ (or vector spaces over $\mathbb{F}_q$ whose bases are subsets of $N \cup \{ \ast \}$ for $VI$) and morphisms injections (or linear injections for $VI$). Clearly, there is a natural isomorphism $\varphi: C \to D$. We then define $\iota = \varphi \circ \iota'$.

Given a finitely generated $C$-module $V$, the assignment $V \to V \circ \iota$ induces an endo-functor $\Sigma: C \module \to C \module$, which is called a \emph{shift functor}. Furthermore, in \cite{GL} it has been shown that $\Sigma$ is a \emph{genetic shift functor}. That is, we have the following result, whose proof can be found in \cite[Section 5]{GL}.

\begin{lemma} \label{genetic shift functor}
For every $x \in \Ob (C)$, one has $\Sigma (kC e_x) \cong P \oplus Q$, where $P$ and $Q$ are finitely generated projective $kC$-modules, $P$ is generated by its values on objects isomorphic to $x$, and $Q$ is generated by its value on some objects $y$ satisfying $|x| = |y| + 1$, where $|x|$ means the size of the finite set $x$ for all categories except $VI$, and the dimension of $x$ for $VI$.
\end{lemma}

\subsection{Homology theory of $C$-modules}

Now let us define a uniform homology theory for $C$-modules. Let $V$ be a finitely generated $C$-module, and let $x$ be an object. Then $V_x$ is a finite dimensional vector space over $k$. For every object $y \prec x$ which means that $y \preccurlyeq$ but $y$ is not isomorphic to $x$, there are morphisms $\alpha: y \to x$, which give rise to linear maps $V_{\alpha}: V_y \to V_x$. The \emph{zeroth homology group} of $V$ is defined to be
\begin{equation*}
H_0^{C}(V) = \bigoplus _{x \in \Ob (C)} V_x / \Big{(} \sum_{\begin{matrix} y \prec x \\ \alpha: y \to x \end{matrix}} V_{\alpha} (V_y) \Big{)},
\end{equation*}
which is a finitely generated $C$-module as well. The \emph{generating degree} of $V$ is defined to be
\begin{equation*}
\gd(V) = \sup \{|x| \mid (H_0^{C}(V))_x \neq 0 \}.
\end{equation*}
When the above set is empty (this happens if and only if $V = 0$), we set $\gd(V) = -1$.

Taking zeroth homology groups of $C$-modules gives a functor $H_0^{C}: C \module \to C \module$. It is a left exact functor, and one can define its right derived functors $H_s^{C}: C \module \to C \module$ for $s \in N$. Correspondingly, we define the $s$-th \emph{homology group} $H_s^{C}(V)$ and the $s$-th \emph{homological degree}
\begin{equation*}
\hd_s(V) = \sup \{|x| \mid (H_s^{C}(V))_x \neq 0 \}.
\end{equation*}
Again, if the above set is empty, we let $\hd_s(V) = -1$.

For every natural number $n$, the category $C$ has a full subcategory $Cn$, whose objects are $x \in \Ob (C)$ such that $|x| \leqslant n$. Note that the category algebra $kCn$ is not finite dimensional, but is Morita equivalent to a finite dimensional algebra since $C_n$ is equivalent to a finite category (for instances, the ones described in Table 1.1). Therefore, the category of finitely generated $Cn$-modules is also abelian. We call $Cn$ a \emph{finite truncation} of $C$. Since $Cn$ is also a weakly directed category with respect to the preorder inherited from that on $\Ob (C)$, similarly, one can define a homology theory for $kCn$-modules.

\section{Lift and Restriction functors}

In this section we introduce two important natural functors between $C \module$ and $Cn \module$: the \emph{lift} functor and the \emph{restriction} functor, and prove certain important properties.

For every $V \in Cn \module$, define a $C$-module
\begin{equation*}
\tilde{V} = \bigoplus _{x \in \Ob (C)} \tilde{V}_x
\end{equation*}
by letting
\begin{equation*}
\tilde{V}_x = \begin{cases}
V_x, & \text{if } |x| \leqslant n;\\
0, & \text{otherwise.}
\end{cases}
\end{equation*}
Clearly, the assignment
\begin{equation*}
L: Cn \module \to C \module
\end{equation*}
sending $V$ to $\tilde{V}$ is a functor, called the \emph{lift functor}. Furthermore, the functor $L$ is exact.

By this lift process, any $Cn$-module can be viewed as a $C$-module.

Since $Cn$ is a full subcategory of $C$, there is a natural restriction functor $R: C \module \to Cn \module$. Explicitly, given a $C$-module $W$, one can define a $Cn$-module
\begin{equation*}
RW = \bigoplus _{x \in \Ob (Cn)} W_x.
\end{equation*}
The restriction functor $R$ is also exact. Furthermore, applying \cite[Lemma 4.2.19]{Xu}, one has:

\begin{lemma}
The restriction functor $R$ preserves projective modules.
\end{lemma}

\begin{proof}
The conclusion follows from the observation that $Cn$ is an ideal of $C$.
\end{proof}

An immediate consequence of the above lemma is:

\begin{corollary}
Let $W$ be an $C$-module, and let $P^{\bullet} \to W \to 0$ be a projective resolution of $W$. Then $RP^{\bullet} \to RW \to 0$ is a projective resolution of the $Cn$-module $RW$.
\end{corollary}

Now we want to describe the behavior of homology functors under lift and restriction.

\begin{lemma} \label{key lemma}
Let $V$ be a finitely generated $C$-module, and $W$ be a finitely generated $Cn$-module. Then one has:
\begin{enumerate}
\item $R(H_0^{C} (V)) = H_0^{Cn} (RV)$;

\item $R(H_0^{C} (LW)) = H_0^{Cn} (W)$.
\end{enumerate}
\end{lemma}

\begin{proof}
For every object $x \in \Ob (Cn)$, one has
\begin{equation*}
(R(H_0^{C} (V)))_x = (H_0^{C} (V))_x
\end{equation*}
by the definition of the restriction functor $R$. However, from the definition of zeroth homology groups
\begin{equation*}
H_0^{C}(V) = \bigoplus _{x \in \Ob (C)} V_x / \Big{(} \sum_{\begin{matrix} y \prec x \\ \alpha: y \to x \end{matrix}} V_{\alpha} (V_y) \Big{)},
\end{equation*}
one can find the value of $H_0^{C}(V)$ on $x$ is completely determined by objects $y$ with $y \prec x$ and all morphisms from $y$ to $x$. That is, it has nothing to do with $V_z$ with $|z| > n$. Therefore, one can see that
\begin{equation*}
(H_0^{C} (V))_x = (H_0^{Cn} (RV))_x.
\end{equation*}
Consequently, for every $x \in \Ob (Cn)$, one has
\begin{equation*}
(R(H_0^{C} (V)))_x = (H_0^{Cn} (RV))_x.
\end{equation*}
This proves the first statement.

To show the second identity, we apply the first statement to $LW$ and note that $RLW = W$.
\end{proof}

Now let us turn to projective covers of $C$-modules, which have been described in literature, see for instances \cite{CEF, LY}.

Let $V$ be a finitely generated $C$-module. For every object $x$,
\begin{equation*}
(H_0^{C}(V))_x = V_x / \Big{(} \sum_{\begin{matrix} y \prec x \\ \alpha: y \to x \end{matrix}} V_{\alpha} (V_y) \Big{)}
\end{equation*}
is a finite dimensional $kC(x,x)$-module (since $V$ is a finitely generated $C$-module, $V_x$ is finite dimensional). Note that $kC(x, x)$ is a finite dimensional group algebra. Therefore, one can find a projective cover $W_x \to (H_0^{C}(V))_x$ of $kC(x, x)$-modules for each $x \in \Ob (C)$. Now let
\begin{equation*}
P = \bigoplus _{x \in \Ob (C)} kC e_x \otimes_{kC(x, x)} W_x.
\end{equation*}
Then $P$ is a projective cover of $V$ as $C$-modules. Note that
\begin{equation*}
H_0^{C} (P) = \bigoplus _{x \in \Ob (C)} W_x.
\end{equation*}

Projective covers of finitely generated $Cn$-modules can be constructed in a similar way using the homology theory of $Cn$-modules.

The following proposition gives us a bridge between minimal projective resolutions of $C$-modules and minimal projective resolutions of $Cn$-modules.

\begin{proposition}
Let $V$ be a finitely generated $C$-module. If $P \to V$ is a projective cover of $C$-modules, then $RP \to RV$ is a projective cover of $Cn$-modules as well.
\end{proposition}

\begin{proof}
By the above description of projective covers, a finitely generated projective $Cn$-module $Q$ is a projective cover of $RV$ if and only if for every $x \in \Ob (Cn)$, $(H_0^{Cn} (Q))_x$ is a projective cover of $(H_0^{Cn} (RV))_x$ as $kC(x, x)$-modules. By Lemma \ref{key lemma}, this is equivalent to saying that $(H_0^{Cn} (Q))_x$ is a projective cover of $(R(H_0^{C} (V)))_x = (H_0^{C} (V))_x$.

Since $P$ is a projective cover of $V$, one knows that $(H_0^{C}(P))_x$ is a projective cover of $(H_0^{C}(V))_x$. Furthermore, since $RP$ is a projective $Cn$-module, by the conclusion in the previous paragraph, it suffices to show that
\begin{equation*}
(H_0^{Cn} (RP))_x = (H_0^{C}(P))_x.
\end{equation*}
But this follows immediately from Lemma \ref{key lemma} and the trivial identity
\begin{equation*}
(R(H_0^{C} (P)))_x = (H_0^{C}(P))_x.
\end{equation*}
\end{proof}

As a corollary, we have:

\begin{corollary} \label{restriction preserves minimal resolution}
Let $V$ be a finitely generated $C$-module. If $P^\bullet \to V \to 0$ is a minimal projective resolution of $V$, then $RP^\bullet \to RW \to 0$ is a minimal projective resolution of $RW$.
\end{corollary}

\begin{proof}
This is clear since minimal projective resolutions are constructed by finding projective covers step by step.
\end{proof}

\section{Proof of the main theorems}

In this section we classify the global dimension of $kCn$. For this purpose, we need the following result, which is obtained by Xu in \cite[Theorem 5.3.1]{Xu}.

\begin{theorem}
Let $C$ be a finite EI-category and $k$ be a field. Then $kC$ has finite global dimension if and only if for all $x \in \Ob (C)$, the automorphism group $G_x = C(x, x)$ has order invertible in $k$. In that case, one has:
\begin{equation*}
\gl.dim kC \leqslant \max \{l(C) \mid C \text{ is a linear chain in the poset } \Ob (C) \},
\end{equation*}
where $l(C)$ is the length of $C$.
\end{theorem}

Although $Cn$ is not a finite $EI$-category, it has a skeletal full subcategory whose objects are $[i] = \{1, 2, \ldots, i\}$ or $F_q^{\oplus i}$, $i \leqslant n$, where we set $[0] = \emptyset$. Therefore, from this result, one gets

\begin{lemma} \label{classify global dimension}
Let $k$ be a field, $G$ be a finite group, and $d$ be a positive integer. One has:
\begin{enumerate}
\item for $C = OI$ or $C = OI_d$, $\gl.dim kCn \leqslant n$;
\item for $C = FI$ or $C = FI_d$, $\gl.dim kCn = \infty$ if $\Char k \mid n!$, and $\gl.dim kCn \leqslant n$ otherwise;
\item for $C = FI_G$, $\gl.dim kCn = \infty$ if $\Char k \mid n!|G|^n$, and $\gl.dim kCn \leqslant n$ otherwise;
\item for $C = OI_G$, $\gl.dim kCn = \infty$ if $\Char k \mid |G|$, and $\gl.dim kCn \leqslant n$ otherwise;
\item for $C = VI$, $\gl.dim kCn = \infty$ if $\Char k \mid (q^n-1)(q^n - q)\ldots (q^n-q^{n-1})$, and $\gl.dim kCn \leqslant n$ otherwise.
\end{enumerate}
\end{lemma}

\begin{proof}
We note that every linear chain in the poset $\Ob (Cn)$ has length at most $n$. Therefore, by the previous theorem, we only need to determine the situation that for every object $x \in \Ob (C)$, the automorphism group $C(x, x)$ has order invertible in $k$.
\begin{itemize}
\item For $OI$ or $OI_d$, $C(x, x)$ is a trivial group.
\item For $FI$, $C(x, x)$ has order $|x|!$, which is invertible in $k$ for every $x \in \Ob (Cn)$ if and only if $\Char k \nmid n!$.
\item For $FI_G$, $C(x, x)$ has order $|x|!G^{|x|}$, which is invertible in $k$ for every $x \in \Ob (Cn)$ if and only if $\Char k \nmid n!|G|^n$.
\item For $OI_G$, $C(x, x)$ has order $|G|^n$, which is invertible in $k$ if and only if $\Char k \nmid |G|$.
\item For $VI$, $C(x, x)$ has order $(q^{|x|} - 1) \ldots (q^{|x|} - q^{|x|-1})$, which is invertible in $k$ for every $x \in \Ob (Cn)$ if and only if $\Char k \nmid (q^n-1) \ldots (q^n - q^{n-1})$.
\end{itemize}
The conclusion then follows.
\end{proof}

To prove Theorem 1.1, we need the following result, which is \cite[Corollary 5.1]{L}.

\begin{lemma}
Let $V$ be a finitely generated $Cn$-module, and let $P^{\bullet} \to LV \to 0$ be a minimal projective resolution. Then for every $s \in N$, one has
\begin{equation*}
\gd(P^s) \leqslant \max \{ |x| \mid V_x \neq 0 \} + s.
\end{equation*}
\end{lemma}

\begin{proof}
By \cite[Corollary 5.1]{L}, one has
\begin{equation*}
\hd_s(LV) \leqslant \max \{ |x| \mid V_x \neq 0 \} + s
\end{equation*}
for $s \in N$. By \cite[Corollary 2.10]{LY}, \footnote{In this corollary the authors only proved the conclusion for $C = FI$. However, a careful observation convinces the reader that their argument actually works for all categories $C$ studied in this paper since it does not rely on any specific combinatorial properties uniquely possessed by $FI$.} we also have
\begin{equation*}
\gd(P^s) \leqslant \max \{\hd_i(V) \mid 0 \leqslant i \leqslant s \}.
\end{equation*}
The conclusion follows.
\end{proof}

Now we consider a special $Cn$-module $M$ defined as follows:
\begin{equation*}
M_x = \begin{cases}
k, & |x| = 0;\\
0, & 0 < |x| \leqslant n.
\end{cases}
\end{equation*}
For this module, we have:

\begin{lemma} \label{special simple module}
Let $k$ be field such that $|C(x,x)|$ is invertible in $k$ for every $x \in Cn$. Then the projective dimension $\pd_{kCn} (M) \geqslant n$.
\end{lemma}

\begin{proof}
Let $P^{\bullet} \to LM \to 0$ be a minimal projective resolution of the $C$-module $LM$.

\textbf{Step 1}: We claim that for every $s \leqslant n$, $P^s$ is either 0 or is generated by its values on objects $x$ with $|x| = s$. We prove this statement by induction. For $s = 0$, the conclusion holds trivially. Suppose that it is true for $s = i$, and let us consider $P^{i+1}$. Without loss of generality we may assume that $P^{i+1} \neq 0$.

We have a short exact sequence $0 \to Z^{i+1} \to P^i \to Z^i \to 0$. Since $P^{\bullet} \to LV \to 0$ is a minimal projective resolution and $P^{i+1} \neq 0$, $P^i$ is nonzero as well, and hence must be generated by its values on objects $x$ with $|x| = i$. This is the same for $Z^i$. Therefore,
\begin{equation*}
H_0^{C} (Z^i) = \bigoplus _{|x| = i} (H_0^{C} (Z^i))_x
\end{equation*}
since the values of $H_0^{C} (Z^i)$ on other objects are all zero.

Note that under the given assumption, $kC(x, x)$ is a semisimple algebra for every $x$ with $|x| = i$. Therefore, by the construction of projective covers described in the previous section,
\begin{equation*}
P^i \cong \bigoplus _{|x| = i} kC e_x \otimes _{kC(x, x)} (H_0^{C} (Z^i))_x.
\end{equation*}
Consequently, $Z^{i+1}_x = 0$ for all objects $x$ with $|x| = i$. It is also clear that $Z^{i+1}_x = 0$ for all objects $x$ with $|x| < i$ since it is a submodule of $P^i$. From the surjection $P^{i+1} \to Z^{i+1} \to 0$ one deduces that $P^{i+1}$ is generated by its values on some objects $x$ with $|x| \geqslant i+1$. But from the previous lemma one also has $\gd(P^{i+1}) \leqslant i + 1$. Consequently, $P^{i+1}$ is generated by its values on objects $x$ with $|x| = i+1$. The claimed statement follows by induction.

\textbf{Step 2}: We claim that $P^s$ is nonzero for every $s \leqslant n$, and use contradiction to prove it. Let $s$ be the minimal number in $[n]$ such that $P^s = 0$. Then for every $i < s$, $P^i \neq 0$, and by the conclusion of Step 1, $P^i$ is generated by its values on objects $x$ with $|x| = i$. Furthermore, for $i \geqslant s$, $P^i = 0$ since $P^{\bullet} \to LV \to 0$ is a minimal projective resolution. Therefore, one gets the following resolution
\begin{equation*}
0 \to P^{s-1} \to \ldots \to P^1 \to P^0 \to LM \to 0.
\end{equation*}
Applying the genetic shift functor $\Sigma$ one obtains an exact sequence
\begin{equation*}
0 \to \Sigma P^{s-1} \to \ldots \to \Sigma P^1 \to \Sigma P^0 \to \Sigma LM = 0 \to 0
\end{equation*}
which must split since every term is projective. But this is impossible. Indeed, by Lemma \ref{genetic shift functor}, $\Sigma P^{s-1}$ has a direct summand whose generating degree is $s-1$, while the generating degrees of all other terms in this sequence is strictly less than $s-1$.

\textbf{Step 3}: Now applying the restriction functor $R$ to $P^{\bullet} \to LM \to 0$ we obtain a minimal projective resolution $RP^{\bullet} \to RLM = M \to 0$ by Corollary \ref{restriction preserves minimal resolution}. For every $s \leqslant n$, $RP^s \neq 0$ since by the conclusions of Steps 1 and 2, $P^s$ is nonzero and is generated by its values on objects $x$ with $|x| = s$. Therefore, $\pd_{kCn} (M) \geqslant n$.
\end{proof}

\begin{remark}
When $k$ is a field of characteristic 0, it has been shown in \cite{GL} that $kC$ is a Koszul algebra; that is, every simple $kC$-module has a linear projective resolution. This trivially implies the conclusion of Step 1 in the proof. However, for fields with a positive characteristic, $kC$ is not Koszul, and hence we have to find a new method to show this result.
\end{remark}

Let $x$ be an object in $Cn$, and let $S_x$ be an arbitrary simple $kCn$-module supported on $x$. That is, $S_x$ restricted to an object $y$ isomorphic to $x$ is a simple $kCn(y, y)$-module, and restricted to objects not isomorphic to $x$ is 0. Applying the same argument in Lemma \ref{special simple module}, one has:

\begin{corollary}
Let $k$ be a field such that $|Cn(y, y)|$ is invertible in $k$ for every $y \in \Ob (Cn)$. Then for $x \in \Ob (Cn)$ and a simple $kC_n$-module $S_x$ supported on $x$, there is a linear minimal projective resolution
\begin{equation*}
0 \to P^{n-|x|} \to P^{n - |x|-1} \to \ldots \to P^0 \to S_x \to 0
\end{equation*}
such that $P^i$ is generated by its values on objects $y$ satisfying $|y| = i + |x|$ for $0 \leqslant i \leqslant n-|x|$. In particular, $kC_n$ is a Koszul algebra.
\end{corollary}

\begin{proof}
We note that the category $Cn$ is a graded category by letting the degree of a morphism $\alpha: x \to y$ be $|y| - |x|$. By slightly modifying the proof of the previous lemma, one can show the existence of such a linear minimal projective resolution for $S_x$. Furthermore, by \cite[Proposition 4.3]{W}, every simple $kCn$-module is isomorphic to $S_x$ supported on a certain object $x \in \Ob (Cn)$. The conclusion follows.
\end{proof}

\begin{remark}
When $k$ is a field of characteristic 0, the Koszulity of $kCn$ follows from \cite[Theorem 1.1]{GL} and \cite[Corollary 5.12]{GL}. But these results do not apply to the positive characteristic case.
\end{remark}

Now we are ready to prove Theorem \ref{main theorem}.

\begin{proof}
In the situation that there is an object $x$ in $Cn$ such that the order of $C(x, x)$ is not invertible in $k$, by Lemma \ref{classify global dimension}, $\gl.dim kCn = \infty$. Otherwise, $\gl.dim kCn \leqslant n$ by the same lemma, and $\gl.dim kCn \geqslant n$ by Lemma \ref{special simple module}. The conclusion follows.
\end{proof}

\end{document}